000.README.txt
nohypertex
\documentclass[12pt]{article}
\usepackage{amsmath,amscd,amsthm,amsfonts,amssymb}
\begin{document}
\newtheorem{thm}{Theorem}
\newtheorem{cor}{Corollary}
\newtheorem{lem}{Lemma}
\newtheorem{slem}{Sublemma}
\newtheorem{prop}{Proposition}
\newtheorem{defn}{Definition}
\newtheorem{conj}{Conjecture}
\newtheorem{ques}{Question}
\newtheorem{claim}{Claim}
\newcounter{constant} 
\newcommand{\newconstant}[1]{\refstepcounter{constant}\label{#1}}
\newcommand{\useconstant}[1]{C_{\ref{#1}}}
\title{Desingularization of branch points of minimal disks in $\mathbb{R}^4$}
\author{Marina Ville}
\date{ }
\maketitle
\begin{abstract}
We deform a minimal disk in $\mathbb{R}^4$ with a branch point into symplectic minimally immersed disks with only transverse double points.
\end{abstract}  
\section{Introduction}
This paper continues the study of branch points of minimal disks in $\mathbb{R}^4$ and their knots which was started in [Vi], [S-V1] and [S-V2]. Near the branch point, the disk is symplectic for two different symplectic structures, one for each  orientation in $\mathbb{R}^4$. For each of these symplectic structures, we show that the branched disk can be deformed into symplectic minimally immersed disks with only transverse double points. If the branched disk is topologically embedded, this can be done without changing the transverse knot type of the boundary knot and the number of the double points of the immersed disks is given by the self-linking number of the transverse knot.
\section*{Acknowledgements}
The author is very grateful to Marc Soret for being a great long-time partner in the study of branch points.
\section{Preliminairies}
\subsection{Branch points}\label{definition of a branch point}
Let $F:\mathbb{D}\longrightarrow\mathbb{R}^4$ be a map. A point $p\in\mathbb{D}$ is a branch point of $F$ if we can find a coordinate system $(x_i)$ around $F(p)$ in which the map is written as
\begin{equation}\label{branch point}
F_1(z)+iF_2(z)=z^N+o(|z|^N)\ \ \ \ \ F_3(z)+iF_4(z)=o(|z|^N)
\end{equation}
where $F_i(z)$ denotes the $i$-th component of $F(z)$ in the coordinate system $(x_i)$.\\
Here and throughout this paper, $p$ is identifies with $0$ and $F(p)$ is identified with $(0,...,0)$ in $\mathbb{R}^4$.
The quantity $N-1$ is called the {\it branching order} of $F$ at $p$.
\subsection{The Grassmannian}
We denote by $G_2^+(\mathbb{R}^4)$ the Grassmannian of oriented $2$-planes in $\mathbb{R}^4$. An oriented $2$-plane $P$ can be viewed as the $2$-vector $e_1\wedge e_2$ where  $(e_1, e_2)$ is a positive orthonormal basis of $P$. Thus $G_2^+(\mathbb{R}^4)$ is embedded in $\Lambda^2(\mathbb{R}^4)$; if we write $P$ as a $2$-vector we can define
\begin{equation}\label{splitting of G2}
H=\frac{1}{\sqrt{2}}(P+\star P)\ \ \ \ \ K=\frac{1}{\sqrt{2}}(P-\star P)
\end{equation}
where $\star:\Lambda^2(\mathbb{R}^4)\longrightarrow \Lambda^2(\mathbb{R}^4)$ is the Hodge operator ([Be] or [Jo] p. 82).\\
The $2$-vector $H$ (resp. $K$) defined  in (\ref{splitting of G2}) belongs to the unit sphere of $\Lambda^{+}(\mathbb{R}^4$) (resp. $\Lambda^{-}(\mathbb{R}^4$)) and we derive an identification
\begin{equation}\label{Grassmannian}
G_2^+(\mathbb{R}^4)\cong\mathbb{S}(\Lambda^+(\mathbb{R}^4))\times\mathbb{S}(\Lambda^-(\mathbb{R}^4))
\end{equation}
\subsection{The Gauss map}
If $F:\mathbb{D}\longrightarrow \mathbb{R}^4$ is an immersion, we derive two Gauss maps
\begin{equation}\label{Gauss maps}
\gamma_+:\mathbb{D}\longrightarrow \mathbb{S}(\Lambda^+(\mathbb{R}^4)),\ \ \ \
\gamma_-:\mathbb{D}\longrightarrow \mathbb{S}(\Lambda^-(\mathbb{R}^4))
\end{equation}
as follows. We let $z\in  \mathbb{D}$ and let $P$ be the oriented tangent plane  $F_\star(T_z\mathbb{D})$; the orientation on $P$ is defined via $F$ by the orientation on $\mathbb{D}$. Using (\ref{splitting of G2}), we write 
$P=\frac{1}{\sqrt{2}}(H+K)$ where $H$ (resp. $K$) of $\Lambda^+(\mathbb{R}^4)$ (resp. $\Lambda^-(\mathbb{R}^4)$). We let
\begin{equation}\label{definition des appli de gauss}
H=\gamma_+(z)\ \ \ \ \ \ K=\gamma_-(z)
\end{equation}
Note that there is another way of defining the Gauss map: we write $F$ in components as $F=(F_1,F_2,F_3,F_4)$ and define for every $i=1,...,4$ the complex number
\begin{equation}\label{complexification}
\phi_i=\frac{\partial F_i}{\partial x}-i\frac{\partial F_i}{\partial y}
\end{equation}
Identifying the $2$-spheres $\mathbb{S}(\Lambda^+(\mathbb{R}^4))$ and $\mathbb{S}(\Lambda^-(\mathbb{R}^4))$ with the complex projective line $\mathbb{C}P^1$, we can write ([M-O])
\begin{equation}\label{appli de gauss: droite projective}
\gamma_+=\frac{\phi_3+i\phi_4}{\phi_1-i\phi_2}\ \ \ \ 
\gamma_-=\frac{-\phi_3+i\phi_4}{\phi_1-i\phi_2}
\end{equation}
\subsection{The symplectic structures associated to the branch point} 
The tangent plane at $p$ to $F(\mathbb{D})$ in \S\ref{definition of a branch point} is the plane $P_0$ generated by $\frac{\partial}{\partial x_1}$ and $\frac{\partial}{\partial x_2}$; we orient it by taking $(\frac{\partial}{\partial x_1},\frac{\partial}{\partial x_2})$ to be a positive basis; it is a complex line for two orthogonal complex structures on $\mathbb{R}^4$, one for each orientation. In terms of $2$-vectors, these complex structures are written
\begin{equation}\label{complex structure}
H_0=\frac{1}{\sqrt{2}}(P_0+*P_0)\ \ \ \ K_0=\frac{1}{\sqrt{2}}(P_0-*P_0)
\end{equation}
The $2$-vectors $H_0$ and $K_0$ define symplectic forms $\omega_+$ and $\omega_-$ on $\mathbb{R}^4$ as follows
\begin{equation}\label{symplectic structures}
\omega_+(u,v)=<H_0,u\wedge v>\ \ \ \ \omega_-(u,v)=<K_0,u\wedge v>
\end{equation}
for two vectors $u,v\in\mathbb{R}^4$ ($<,>$ denotes the scalar product on $2$-vectors).\\
In a neighbourhood of $p$, a tangent plane $P$ to $F(\mathbb{D})$ is
symplectic for both $\omega_+$ and $\omega_-$, that is, it verifies
\begin{equation}\label{+symplectic tangent planes}
<P,H_0>>0
\end{equation}
\begin{equation}\label{+symplectic tangent planes}
<P,K_0>>0.
\end{equation}
Unlike in the case of a complex curve in a complex surface, there is no preferred orientation associated to a minimal surface, so we consider both these symplectic structures. 
\subsection{The knot of the branch point}\label{The knot of the branch point}
In this section we assume that the map $F$ defined in \ref{definition of a branch point} is a topological embedding.\\
Given a small positive number $\epsilon$, we denote by  $\mathbb{S}_\epsilon$ (resp. $\mathbb{B}_\epsilon$) the sphere (resp. ball) centered at $p$ and of radius $\epsilon$. If $\epsilon$ is small enough, $K^\epsilon=\mathbb{S}_\epsilon\cap F(D)$ is a knot
and $F(\mathbb{D}\cap\mathbb{B}_\epsilon)$ is homeomorphic to the cone on $K^\epsilon$ (cf. [S-V1] where this construction follows from [Mi]).
\subsection{The braid defined by the knot $K^\epsilon$ and its writhe number}
The knot $K^\epsilon$ is naturally presented as a braid with $N$ strands in the $3$-sphere (cf. [Vi]); the axis of this braid is the great circle in the normal plane at $0$, that is the plane which is orthogonal to the tangent plane at $0$ generated by $\frac{\partial}{\partial x_1}$ and $\frac{\partial}{\partial x_2}$. The {\it algebraic crossing number} of this braid is
\begin{equation}\label{definition of the writhe number}  
e(K^\epsilon)=lk(K^\epsilon,\hat{K}^\epsilon)
\end{equation}
where $\hat{K}^\epsilon$ is the knot obtained by pushing slightly $K^\epsilon$ in the direction of the axis of the braid.\\
REMARK. In [S-V1], we consider the knot in the cylinder $\{(z_1,z_2)\in\mathbb{C}^2\slash |z_1|=\eta\}$; and in [S-V2] we use the terme {\it writhe} instead of {\it algebraic crossing number}. 

\section{Desingularization of a branch point}
\begin{thm}\label{theorem:single disk}
Let $F:\mathbb{D}\longrightarrow\mathbb{R}^4$ be a minimal map with a branch point as in \S \ref{definition of a branch point}.\\
For some real number $\epsilon>0$ there exists, for $t\in [0,\epsilon)$, a smooth family $F_t^{(+)}:\mathbb{D}\longrightarrow \mathbb{R}^4$ 
(resp. $F_t^{(-)}:\mathbb{D}\longrightarrow \mathbb{R}^4$) of minimal immersions such that\\
1) $F^{(-)}_0=F^{(+)}_0=F$.\\
For every $t$ small enough, \\
2) $F_t$ is an immersion with transverse double points.\\
3) $F_t^{(+)}$ 
(resp. $F_t^{(-)}$) is symplectic w.r.t. $\omega_+$ (resp. $\omega_-$).\\
If $F$ is an embedding, we have\\
4) The numbers $D^{(+)}$, $D^{(-)}$ of double points of  $F_t^{(+)}$, $F_t^{(-)}$ verify
\begin{equation}\label{equation:number of double points in the positive case}
2D^{(+)}=e(K)-(N-1)
\end{equation}
\begin{equation}\label{equation:number of double points in the negative case}
2D^{(-)}=-w(K)-(N-1)
\end{equation}
where $N-1$ is the branching order (cf. \S \ref{definition of a branch point}).
\end{thm}
PROOF OF THEOREM \ref{theorem:single disk}.\\
Each coordinate function $F_i$, $i=1,...,4$ is harmonic, hence there exist four holomorphic functions $f_1,...,f_4$ such that
\begin{equation}\label{equation:formule de F}
F_1+iF_2=f_1+\bar{f_2}\ \ \ \ \ \ \ \ F_3+iF_4=f_3+\bar{f_4}
\end{equation} 
Since $F$ is a conformal map, the $f_i$'s verify (cf. [M-W])
\begin{equation}\label{equation:four functions}
f'_1f'_2+f'_3f'_4=0
\end{equation}
\begin{lem}\label{lemma:branch point}
A point $z_0$ in $\mathbb{D}$ is a branch point if and only if for every $i=1,...,4$
$$f'_i(z_0)=0$$
\end{lem}
\begin{proof}
The point $z_0$ is a branch point if and only if $\frac{\partial F}{\partial x}(z_0)=\frac{\partial F}{\partial y}(z_0)=0$.  Lemma \ref{lemma:branch point} follows from looking at the formulae for the derivatives of $F$ 
\[ \frac{\partial F}{\partial x} = \left( \begin{array}{c}
Re(f'_1+f'_2) \\
Im (f'_1-f'_2) \\
Re(f'_3+f'_4)\\
Im (f'_3-f'_4)\end{array} \right) \] 
\[ \frac{\partial F}{\partial y} = \left( \begin{array}{c}
-Im(f'_1+f'_2) \\
Re(f'_1-f'_2) \\
-Im(f'_3+f'_4)\\
Re(f'_3-f'_4)\end{array} \right) \] 
\end{proof}
We now assume $z_0=0$ which causes no loss of generality.\\
Going back to the assumptions of Th. \ref{theorem:single disk}, we derive the existence of holomorphic functions $\tilde{f}_i$'s and positive integers $n_i$, $i=1,...,4$ such that for every
$i=1,...,4$
\begin{equation}\label{equation:zero}
f_i'=z^{n_i}\tilde{f}_i
\end{equation}
with $\tilde{f}_i(0)\neq 0$. We derive from  that (\ref{equation:four functions}) that
\begin{equation}\label{quatre puissances}
n_1+n_2=n_3+n_4
\end{equation}
Without loss of generality, we assume that $n_1< n_2,n_3,n_4$. It follows from (\ref{quatre puissances}) that $\tilde{f}_3$ and $\tilde{f}_4$ have order smaller than $\tilde{f}_2$. \\
We construct the $F^{(+)}_t$'s and we indicate what to change to construct the $F^{(-)}_t$'s.\\
For $A=(a_0,...,a_{n_1})\in\mathbb{C}^{n_1+1}$ and $B=(b_0,...,b_{n_3})\in\mathbb{C}^{n_3+1}$, we let
\begin{equation}\label{equation:nouvelles fonctions1}
h_1(z,A,B)=(z^{n_1}+\sum_{i=0}^{n_1}a_iz^i)\tilde{f}_1(z)\ \ \
h_2(z,A,B)=z^{n_2-n_3}(z^{n_3}+\sum_{i=0}^{n_3}b_iz^i)\tilde{f}_2(z)
\end{equation}
\begin{equation}\label{equation:nouvelles fonctions2}
h_3(z,A,B)=(z^{n_3}+\sum_{i=0}^{n_3}b_iz^i)\tilde{f}_3(z)\ \ \
h_4(z,A,B)=z^{n_4-n_1}(z^{n_1}+\sum_{i=0}^{n_1}a_iz^i)\tilde{f}_4(z)
\end{equation}
The $h_i$'s are holomorphic and verify (using (\ref{equation:four functions}) and (\ref{quatre puissances}))
\begin{equation}\label{equation:quatre nouvelles fonctions}
h_1h_2+h_3h_4=0
\end{equation}
For $i=1,...,4$, we let 
\begin{equation}\label{equation:nouvelle fonction}
f_i(z,A,B)=\int_0^z h_i(\xi,A,B)d\xi
\end{equation}
The  $f_i(.,A,B)$'s are holomorphic and verify $\frac{\partial f_i}{\partial z}=h_i$.
We let $$F(z,A,B)=(f_1(z,A,B)+\bar{f}_2(z,A,B),f_3(z,A,B)+\bar{f}_4(z,A,B)).$$ It follows from (\ref{equation:quatre nouvelles fonctions}) that for every $(A,B)$, the $F(.,A,B)$'s are minimal maps. We assume that $(A,B)$ belongs to the open dense set $X_1$ of $\mathbb{C}^{n_1+1}\times\mathbb{C}^{n_4+1}$
of the $(A,B)$'s such that the polynomials $z^{n_1}+\sum_0^{n_1}a_iz^i$
and $z^{n_4}+\sum_0^{n_4}b_iz^i$ have distinct roots which are all different from $0$.
It follows from Lemma \ref{lemma:branch point} that for  $(A,B)$ in $X_1$, $F(.,A,B)$ is an immersion. \\
We compute their Gauss maps using (\ref{appli de gauss: droite projective})
and we see that 
\begin{equation}\label{equation:symplectic}
\gamma_+(F(.,A,B))=\frac{h_3(.,A,B)}{h_2(.,A,B)}=z^{n_3-n_2}\frac{\tilde{f}_3}{\tilde{f}_2}=\frac{f'_3}{f'_2}=\gamma_+(F)
\end{equation}
It follows that the $F(.,A,B)$'s are symplectic w.r.t. $\omega_+$.\\
Note that if we want the 
$F(.,A,B)$'s to be symplectic w.r.t. $\omega_-$, we define instead
\begin{equation}\label{equation:nouvelles fonctions negatives1}
h_1(z,A,B)=(z^{n_1}+\sum_{i=0}^{n_1}a_iz^i)\tilde{f}_1(z)\ \ \
h_2(z,A,B)=z^{n_2-n_4}(z^{n_4}+\sum_{i=0}^{n_4}b_iz^i)\tilde{f}_2(z)
\end{equation}
\begin{equation}\label{equation:nouvelles fonctions negatives2}
h_3(z,A,B)=z^{n_3-n_1}(z^{n_1}+\sum_{i=0}^{n_1}a_iz^i)\tilde{f}_3(z)\ \ \
h_4(z,A,B)=(z^{n_4}+\sum_{i=0}^{n_4}b_iz^i)\tilde{f}_4(z)
\end{equation}
We will now use the Transversality  Lemma to prove that that for generic $A,B$, $F(.,A,B)$ has only transverse double points. We do it for the functions defined in (\ref{equation:nouvelles fonctions1}) and (\ref{equation:nouvelles fonctions2}); the proof for (\ref{equation:nouvelles fonctions negatives1}) and (\ref{equation:nouvelles fonctions negatives2}) works identically. \\
We define
$$\Phi:\mathbb{C}^{n_1+1}\times \mathbb{C}^{n_3+1}\times\mathbb{D}\times\mathbb{D}\longrightarrow 
\mathbb{R}^4\times\mathbb{R}^4$$
$$(A,B,z_1,z_2)\mapsto (F(z_1,A,B),F(z_2,A,B))$$
and we prove
\begin{lem}\label{transversalite}
There exists a positive number $\eta$ such that, for every $A$, $B$,
if $z_1\neq z_2$ and $|z_1|<\eta$, $|z_2|<\eta$, then $\Phi$ is transverse to the diagonal $\Delta$ of $\mathbb{R}^4\times\mathbb{R}^4$ at $(A,B,z_1,z_2)$.
\end{lem}
\begin{proof}
We identify $\mathbb{R}^4$ with $\mathbb{C}^{2}$; if $J_0$ is the canonical complex structure on $\mathbb{C}^{2}$, we  introduce a new orthogonal complex structure $J_1$ on $\mathbb{C}^{2}$ defined
\begin{equation} J_1(1,0)=(i,0)\ \ \ \ J_1(0,1)=(0,-i)
\end{equation}
The point of this change is to make $F$ holomorphic w.r.t. $A$ and antiholomorphic w.r.t. $B$.
If we use (\ref{equation:nouvelles fonctions negatives1}) and (\ref{equation:nouvelles fonctions negatives2}), the map $F$ is holomorphic in $A$ and antiholomorphic in $B$ for the standard complex structure $J_0$ so we keep it.\\The diagonal $\Delta$ is a complex subspace of $\mathbb{C}^{4}$ which is generated over the complex numbers by the vectors 
\begin{equation}
\epsilon_1=(1,0,1,0)\ \ \ \ \ \epsilon_2=(0,1,0,1)
\end{equation}
If $i=0,...,n_1$ (resp. $j=0,...,n_3$), we write $a_i$ (resp. $b_j$) in real coordinates
\begin{equation}\label{parties reelles et imaginaires}
a_i=a_i^{(1)}+ia_i^{(2)}\ \ \ \ (\mbox{resp.}\ \ \ \ b_j=b_j^{(1)}+ib_j^{(2)})
\end{equation}
The map $F$ is now holomorphic in $A$ and antiholomorphic in $B$, hence  Lemma \ref{transversalite} will be proved once we prove
\begin{lem}\label{determinant}
$$det(\frac{\partial \Phi}{\partial a_0}, \frac{\partial \Phi}{\partial \overline{b_0}}, \epsilon_1, \epsilon_2)\neq 0$$
the determinant being computed over the complex numbers.
\end{lem}
\begin{proof}
We have
\begin{equation}\label{derivee partielle de Phi}
\frac{\partial \Phi}{\partial a_0}(A,B,z_1,z_2)=\left(\frac{\partial F}{\partial a_0}(z_1,A,B),\frac{\partial F}{\partial a_0}(z_2,A,B)\right)\in\mathbb{C}^2\times\mathbb{C}^2
\end{equation}
For $i=1,2$, we write in Euclidean complex coordinates in $\mathbb{C}^2$,
$$\frac{\partial F}{\partial a_0}(z_i,A,B)=\left(\frac{\partial f_1}{\partial a_0}(z_i,A,B),\frac{\partial f_4}{\partial a_0}(z_i,A,B)\right)$$
\begin{equation}\label{derivee partielle holomorphe de F}
=\left(\int_0^{z_i} \frac{\partial h_1}{\partial a_0}(\xi,A,B)d\xi, \int_0^{z_i} \frac{\partial h_4}{\partial a_0}(\xi,A,B)d\xi\right)
\in\mathbb{C}^2
\end{equation}
by differentiation under the integral sign, hence $$\frac{\partial \Phi}{\partial a_0}(A,B,z_1,z_2)=$$
\begin{equation}\label{derivee partielle holomorphe en 4 dimension}
\left
(\int_0^{z_1} \frac{\partial h_1}{\partial a_0}(\xi,A,B)d\xi, \int_0^{z_1} \frac{\partial h_4}{\partial a_0}(\xi,A,B)d\xi,\int_0^{z_2} \frac{\partial h_1}{\partial a_0}(\xi,A,B)d\xi, \int_0^{z_2} \frac{\partial h_4}{\partial a_0}(\xi,A,B)d\xi
\right)
\end{equation}
$$\mbox{Similarly}\ \ \ \ \frac{\partial \Phi}{\partial \overline{b_0}}(A,B,z_1,z_2)=$$
\begin{equation}\label{derivee partielle antiholomorphe en 4 dimension}
\left
(\int_0^{z_1} \frac{\partial \bar{h}_2}{\partial \overline{b_0}}(\xi,A,B)d\xi,
\int_0^{z_1} \frac{\partial \bar{h}_3}{\partial \overline{b_0}}(\xi,A,B)d\xi,
\int_0^{z_2} \frac{\partial \bar{h}_2}{\partial \overline{b_0}}(\xi,A,B)d\xi,
\int_0^{z_2} \frac{\partial \bar{h}_3}{\partial \overline{b_0}}(\xi,A,B)d\xi\right)
\end{equation} 
We can now compute $$det(\frac{\partial \Phi}{\partial a_0}, \frac{\partial \Phi}{\partial \overline{b_0}}, \epsilon_1, \epsilon_2)=$$
\begin{equation}\label{determinant avec 4 integrales}
\int_{z_2}^{z_1} \frac{\partial h_1}{\partial a_0}(\xi,A,B)d\xi
\int_{z_2}^{z_1} \frac{\partial \bar{h}_3}{\partial \overline{b_0}}(\xi,A,B)d\xi
-
\int_{z_2}^{z_1} \frac{\partial h_4}{\partial a_0}(\xi,A,B)d\xi
\int_{z_2}^{z_1} \frac{\partial \bar{h}_2}{\partial \overline{b_0}}(\xi,A,B)d\xi
\end{equation}
We now compute the derivatives involved:
\begin{equation}\label{derivatives of the hi1}
\frac{\partial h_1}{\partial a_0}(z,A,B)=\tilde{f}_1(z)\ \ \  
\frac{\partial h_2}{\partial b_0}(z,A,B)=z^{n_2-n_3}\tilde{f}_2(z)
\end{equation}
\begin{equation}\label{derivatives of the hi2}
\frac{\partial h_3}{\partial b_0}(z,A,B)=\tilde{f}_3(z)\ \ \
\frac{\partial h_4}{\partial b_0}(z,A,B)=z^{n_4-n_1}\tilde{f}_4(z)\ \ \
\end{equation}
We remind the reader that  $\tilde{f}_1(0)\neq 0$ and $\tilde{f}_3(0)\neq 0$; and on the other hand, $n_2-n_3>0$ and 
$n_4-n_1>0$. This enables us to derive the existence of a positive constant $C$ and of an $\eta>0$ such that,
if $|z_i|<\eta$, for $i=1,2$, then
$$|det(\frac{\partial \Phi}{\partial a_0}, \frac{\partial \Phi}{\partial \overline{b_0}}, \epsilon_1, \epsilon_2)|=|(\ref{determinant avec 4 integrales})|\geq C|z_1-z_2|^2$$
This concludes the proof of Lemmas  \ref{determinant} and \ref{transversalite}.
\end{proof}
We derive from Lemma \ref{transversalite} and from the Transversality Lemma ([G-P]) the existence of a dense subset $X_2$ of the product of the unit balls $\mathbb{B}^{n_1+1}\times \mathbb{B}^{n_3+1}$ such that, if $(A,B)\in X_2$, the map
$$\Phi(A,B,.,.):\{(z_1,z_2)\in\mathbb{D}\times\mathbb{D}\slash z_1\neq z_2\} \longrightarrow \mathbb{R}^4\times \mathbb{R}^4$$ is transversal to $\Delta$. It follows that, if $(A,B)\in X_1\cap X_2$, $F(.,A,B)$ has only transverse double points. To conclude the proof of Th. \ref{theorem:single disk}, we use the Curve Selection Lemma for subanalytic sets (see [B-M],[\L o]). In order to do this, we prove
\begin{lem}
$X_1\cap X_2$ is subanalytic.
\end{lem}
\begin{proof}
The complement of $X_1$ is algebraic so $X_1$ is semialgebraic, hence we just need to show that $X_2$ is subanalytic. We let
$${\mathcal Z}=\{(z_1,z_2,A,B)\in\mathbb{D}\times\mathbb{D}\times\mathbb{B}^{n_1+1}\times\mathbb{B}^{n_3+1}\slash
\|A\|\leq 1, \|B\|\leq 1, z_1\neq z_2\ \ \mbox{and}$$
$$det(\frac{\partial F}{\partial x_1}(z_1,A,B),\frac{\partial F}{\partial y_1}(z_1,A,B),
\frac{\partial F}{\partial x_2}(z_2,A,B),\frac{\partial F}{\partial y_2}(z_2,A,B),\epsilon_1, J_1\epsilon_1,
\epsilon_2, J_1\epsilon_2)^2$$
$$+\|F(z_1,A,B)-F(z_2,A,B)\|^2=0\}.$$
Note that here we are talking of the real determinant in $\mathbb{R}^8$; it follows from its definition that ${\mathcal Z}$ is semianalytic. \\
We let $$\Pi:\mathbb{D}\times\mathbb{D}\times\mathbb{C}^{n_1+1}\times\mathbb{C}^{n_3+1}\longrightarrow
\mathbb{C}^{n_1+1}\times\mathbb{C}^{n_3+1}$$ be the projection. Since ${\mathcal Z}$ is semianalytic, the set $\Pi({\mathcal Z})$ is subanalytic. It follows that $X_2=\mathbb{B}^{n_1+1}\times\mathbb{B}^{n_3+1}\backslash \Pi({\mathcal Z})$ is  subanalytic (cf. the theorem of the complement, [B-M]).
\end{proof}
Since $X_1\cap X_2$ is subanalytic and dense, the Curve Selection Lemma ensures the existence of an analytic path
$$\gamma:[0,\epsilon)\longrightarrow X_1\cap X_2$$
such that $\gamma(0)=0$ and for every $t>0$, $\gamma(t)\in X_1\cap X_2$. We let $F_t^{(+)}(z)=F(z,A(\gamma(t)),B(\gamma(t)))$. If $t>0$, $F_t^{(+)}$ is a minimal immersion with only transverse double points. This proves Th. \ref{theorem:single disk} 1), 2) and 3).
\begin{lem}
$\exists\eta_0\ \mbox{such that}\ \
\forall\eta<\eta_0, \exists t(\eta)\ \mbox{such that}\ \ \forall t, 0<t<t(\eta)$, the knots 
$K_t^{\eta}=F_t^{(+)}(\mathbb{D})\cap\mathbb{S}_\eta$ are transversally isotopic to $K^{\eta}=F(\mathbb{D})\cap\mathbb{S}_\eta$.
\end{lem}
\begin{proof} There is a constant $C$ such that, for $t$ small enough, $|A|\leq C|t|$ and $|B|\leq C|t|$.\\
Also, there exists an $\eta_1$ such that, if $\eta<\eta_1$ and $|F(z)|=\eta$, then 
$$|\rho^N-\eta|\leq\frac{\eta}{10}.$$
So for $\eta<\eta_1$, we let $t(\eta)$ such that, if $t<t(\eta)$ and $|F(z)|=\eta$, then 
\begin{equation}\label{estimee de rho ad hoc}
|\rho^N-\eta|\leq\frac{\eta}{5}
\end{equation}
We let $z=\rho e^{i\theta}\in\mathbb{D}$. We can derive from the construction of $F_t^{(+)}$, the following estimate
\begin{equation}\label{point}
F_t^{(+)}(z)=\rho^Ne^{Ni\theta}X+o(\rho^N)+{\mathcal O}(t)
\end{equation}
where $X=(1,0,0,0)\in\mathbb{R}^4$.\\ We need to say a word of what we mean by the ${\mathcal O}(t)$'s in this paragraph: these terms can contain terms in $\rho^k$, for $k>0$ (and/or later in the proof terms in $\rho^{-k}$). So once $\eta$ is fixed, we can derive $t$ in terms of $\eta$ (hence the notation $t(\eta)$ in the statement of the lemma) such that the ${\mathcal O}(t)$ is as small as we want. The term $o(\rho^N)$ on the other hand, is independent of $t$.\\
The vectors $\frac{1}{N}\rho\frac{\partial}{\partial\rho}$ and $\frac{1}{N}\frac{\partial}{\partial\theta}$ are orthogonal and of the same norm in $\mathbb{D}$; since the $F_t^{(+)}$'s are minimal, the vectors 
\begin{equation}\label{vecteurs tangents}
u_1=\frac{1}{N}\rho\frac{\partial F_t^{(+)}}{\partial\rho}\ \ \ \ u_2=\frac{1}{N}\frac{\partial F_t^{(+)}}{\partial\theta}
\end{equation}
are orthogonal and of the same norm and they generate the plane tangent to $F_t^{(+)}(\mathbb{D})$. We have
\begin{equation}\label{vecteurs tangents1}
u_1=\rho^Ne^{Ni\theta}X+o(\rho^N)+{\mathcal O}(t)\ \ \ u_2=\rho^Ne^{Ni\theta}iX+o(\rho^N)+{\mathcal O}(t)
\end{equation}
The vector $\gamma$ tangent to $K_t^{\eta}$ at $F_t^{(+)}(z)$ is of the form $\gamma=au_1+bu_2$ and verifies
\begin{equation}\label{orthogonality}
<F_t^{(+)}(z),\gamma>=0
\end{equation}
We derive from (\ref{orthogonality}), (\ref{point}) and (\ref{vecteurs tangents1}) that
\begin{equation}
<\gamma,u_1>=\|\gamma\|(o(\rho^{2N})+{\mathcal O}(t))
\end{equation}
Hence (remember that $\|u_1\|=\|u_2\|$)
$$<\gamma,u_2>^2=\|\gamma\|^2<u_1,u_1>^2-<\gamma,u_1>^2=\|\gamma\|^2(\rho^{2N})+o(\rho^{2N})+{\mathcal O}(t))$$
On the other hand, $iF_t^{(+)}(z)-u_2=o(\rho^{N})+{\mathcal O}(t)$ hence
$$<\gamma,iF_t^{(+)}(z)>^2=<\gamma,u_2>^2+\|\gamma\|^2(o(\rho^{2N})+{\mathcal O}(t))=
\|\gamma\|^2(\rho^{2N}+o(\rho^{2N})+{\mathcal O}(t))$$
$$\geq \|\gamma\|^2(\frac{\eta^2}{2}+o(\eta^2)+{\mathcal O}(t))$$
This last estimate is derived from (\ref{estimee de rho ad hoc}).
We derive the existence of a $\eta_0<\eta_1$ such that, $\forall \eta<\eta_0$, $\exists t(\eta)$ such that, if $t<t(\eta)$,
$<\gamma,iF_t^{(+)}(z)>\neq 0$. We conclude that all the $K^{\eta}_t$'s are all transverse and they are all transversally isotopic. 
\end{proof}
Note that the $K^{(\eta)}_t$'s are transverse w.r.t. the contact structures associated to {\it both} the symplectic structures. By contrast, the disks $F_t^{(+)}(\mathbb{D})$ and $F_t^{(-)}(\mathbb{D})$ are not symplectic for both structures.\\
The number $D_t$ of transverse double points $F_t^{(+)}$ is given by ([H-H])
\begin{equation}\label{number of double points}
2D_t=sl(K)+1=e(K)-(N-1)
\end{equation}
\end{proof}

\footnotesize{Marina.Ville@lmpt.univ-tours.fr\\LMPT,
 Universit\'e de Tours
 UFR Sciences et Techniques
 Parc de Grandmont
 37200 Tours, FRANCE}

\end{document}